\documentclass[12pt]{amsart}

\usepackage[margin=1in]{geometry}

\usepackage{amscd,latexsym,amsthm,amsfonts,amssymb,amsmath,amsxtra}
\usepackage[mathscr]{eucal}
\usepackage{pictexwd,dcpic}
\usepackage[normalem]{ulem}
\usepackage{hyperref}
\usepackage{stmaryrd}
\usepackage{caption}
\usepackage{xcolor}

\renewcommand\theequation{\thesection.\arabic{equation}}

\newcommand{\BA}{{\mathbb {A}}}

\newcommand{\BC}{{\mathbb {C}}}

\newcommand{\CF}{{\mathcal {F}}}

\newcommand{\CM}{{\mathcal {M}}}

\newcommand{\CP}{{\mathcal {P}}}

\newcommand{\CS}{{\mathcal {S}}}

\newcommand{\Fg}{{\mathfrak {g}}}

\newcommand{\Fu}{{\mathfrak {u}}}

\newcommand{\GL}{{\mathrm{GL}}}

\newcommand{\GSO}{{\mathrm{GSO}}}

\newcommand{\SL}{{\mathrm{SL}}}

\newcommand{\SO}{{\mathrm{SO}}}

\newcommand{\Sp}{{\mathrm{Sp}}}
\newcommand{\Spin}{{\mathrm{Spin}}}
\newcommand{\GSpin}{{\mathrm{GSpin}}}

\newcommand{\pair}[1]{\langle {#1} \rangle}

\newcommand{\back}{\backslash}

\newtheorem{thm}{Theorem}[section]

\newtheorem{lem}[thm]{Lemma}

\newtheorem{prop}[thm]{Proposition}
\newtheorem {conj}[thm]{Conjecture}

\newtheorem {ques/conj}[thm]{Question/Conjecture}

\newtheorem{defn}[thm]{Definition}
\newtheorem{rmk}[thm]{Remark}

\makeatletter

\newcommand{\Rmnum}[1]{\expandafter\@slowromancap\romannumeral #1@}
\makeatother

\begin{document}
\renewcommand{\theequation}{\arabic{equation}}
\numberwithin{equation}{section}

\title{A relative trace formula identity for non-tempered spherical varieties}

\author{Chen Wan}
\address{Department of Mathematics \& Computer Science\\
Rutgers University–Newark\\
Newark, NJ 07102, USA}
\email{chen.wan@rutgers.edu}

\begin{abstract}
In this paper, motivated by some previous works in residue method and the recent theory of the relative Langlands duality, we prove a relative trace formula identity that compares the period integral of non-tempered spherical varieties with the period integral of a tempered spherical varieties associated to a Levi subgroup. This
allows us to incorporate numerous relative trace formula comparisons studied during the
last four decades under the relative Langlands duality framework. We will also propose a conjectural comparison for general non-tempered Hamiltonian spaces.
\end{abstract}



\maketitle

\section{Introduction}
Let $k$ be the number field, $\BA=\BA_k$ be the ring of adeles, and $G$ be a split connected reductive group defined over $k$. Let $B$ be a Borel subgroup of $G$, we say a subgroup $H$ of $G$ is spherical if $B$ admits an open orbit in $X=G/H$. If this is the case, we will say $X$ is a spherical variety of $G$. We say the spherical variety $X$ is reductive (resp. split) if $H$ is reductive (resp. split). Following the structure theory of spherical varieties, we use $P(X)$ to denote the stabilizer of the open Borel orbit in $X$. It is a parabolic subgroup of $G$ and we use $L(X)$ to denote its Levi subgroup. We say the spherical variety is tempered if $P(X)=B$.

One of the most important object in the relative Langlands program is to study the period integrals associated to spherical varieties \footnote{if $H$ has nontrivial intersection with the center $Z_G$ of $G$, then one also need to modulo $Z_G(\BA)\cap H(\BA)$ in the period integral} (here for any group $Y$ defined over $k$ we use $[Y]$ to denote $Y(k)\back Y(\BA)$)
\begin{equation}\label{defn period integral}
\CP_H(\phi):=\int_{[H]} \phi(h)dh
\end{equation}
where $\phi$ is an automorphic form on $G(\BA)$. The goal of this paper is to study the period integral for non-tempered split reductive spherical varieties. When the spherical variety $X$ is not tempered, by a theorem of Sakellaridis (Theorem 3 of \cite{P}), the period integral would be vanishing for all generic automorphic forms. Hence one needs to study the period for residue representations.

One way to study such a period is to compare the relative trace formula associated to this period with another relative trace formula associated to a tempered spherical variety $X_L=L/H_L$ of a certain Levi subgroup $L$ of $G$. The first example of such is due to Jacquet and Rallis \cite{JR} in which they studied the period integral for $\GL_{2n}/\Sp_{2n}$ by comparing it with the period integral of $\GL_n\times \GL_n/\GL_n$. Later their idea was adapted by Jiang, Mao, Rallis in a sequence of papers \cite{J}, \cite{JMR}, \cite{MR} for the models $(G_2,\SL_3),(\SO_7,G_2),(\SO_8,G_2)$ and $(\Sp_{4n},\Sp_{2n}\times \Sp_{2n})$.

In this paper, with the help of the theory of relative Langlands duality \cite{BSV}, we will give a conceptual explanation of those comparisons and we will extend it to all the split reductive spherical varieties. Moreover, we will also propose a conjectural comparison for general hyperspherical Hamiltonian spaces.

\subsection{The main results}
Let $X=G/H$ be a split reductive spherical variety. We are going to specify a Levi subgroup $L$ of $G$ as well as a tempered reductive spherical subgroup $H_L$ of $L$. If $X$ is tempered, then $(L,H_L)=(G,H)$. If $X$ is not tempered, in the following two tables we list all the non tempered spherical varieties as well as the corresponding $(L,H_L)$. In the next subsection, we will give a conceptual explanation of $(L,H_L)$ from the point of view of the relative Langlands duality. Note that $H_L$ is also a Levi subgroup of $H$. The first table contains all the non-tempered cases with $G$ not of Type $F$ and $E$, the second table contains the Type $F$ and $E$ cases. Here we have used the classification of spherical varieties in \cite{BP}.

\begin{figure}[h!]\leftskip-2cm
	\begin{tabular}{| c | c | c | c | c|}
		\hline
		\textnumero & $G$ & $H$ & $L$ & $H_L$ \\
		\hline
		 {1} & $\GL_{2a+2k}$  & $\GL_a\times \GL_{a+2k}$ & $\GL_{2a}\times \GL_{1}^{2k}$  & $\GL_a\times \GL_a\times \GL_{1}^{2k}$   \\
		\hline
         {2} & $\GL_{2a+2k+1}$ & $\GL_a\times \GL_{a+2k+1}$ & $\GL_{2a+1}\times \GL_{1}^{2k}$  & $\GL_{a+1}\times \GL_a\times \GL_{1}^{2k}$ \\
        \hline
         {3} & $\GL_{2n}$ & $\Sp_{2n}$ & $\GL_n\times \GL_n$  & $\GL_n$ \\
        \hline
         {4} & $\Sp_{4m+2k}$ & $\Sp_{2m}\times \Sp_{2m+2k}$ & $\GL_{2m}\times \GL_{1}^{k}$ & $\GL_{m}^{2}\times \GL_{1}^{k}$ \\
        \hline
         {5} & $\Sp_{2n}$ & $\Sp_{2n-2}\times \GL_1$ & $\Sp_4\times \GL_{1}^{n-2}$ & $\Sp_2\times \GL_{1}^{n-1}$ \\
        \hline
         {6} & $\SO_{2m+2k+1}$ & $\SO_{m}\times \SO_{m+2k+1}$ &  $\SO_{2m+1}\times \GL_{1}^{k}$ & $\SO_m\times \SO_{m+1}\times \GL_{1}^{k}$ \\
        \hline
         {7} & $\SO_{2m+2k}$ & $\SO_{m}\times \SO_{m+2k}$ & $\SO_{2m+2}\times \GL_{1}^{k-1}$ & $\SO_{m}\times \SO_{m+2}\times \GL_{1}^{k-1}$ \\
        \hline
         {8} & $\SO_{4n}$ & $\GL_{2n}$ & $\GL_{2n}$ & $\GL_n\times \GL_n$ \\
        \hline
         {9} & $\SO_{4n+2}$ & $\GL_{2n+1}$ & $\GL_{2n}\times \GL_1$ & $\GL_{n}^{2}\times \GL_1$ \\
        \hline
         {10} & $\SO_7$ & $G_2$ &  $\GL_2\times \SO_3$ & $\GL_2$ \\
        \hline
         {11} & $\SO_8$ & $G_2$ & $\GL_2\times \SO_4$ & $\GL_2$ \\
        \hline
         {12} & $G_2$ & $\SL_3$ & $\GL_2$ & $\GL_1\times \GL_1$ \\
        \hline
         {13} & $\GSO_{10}$ & $\GSpin_7$ & $\GSO_6\times \GL_2$ & $\GL_2\times \GL_2$ \\
        \hline
         {14} & $\SO_9$ & $\Spin_7$ & $\SO_5\times \GL_2$ & $\Spin_3\times \GL_2$ \\
        \hline

         {15} & $\GL_{2n+2}\times \GL_2$ & $\GL_{2n}\times \GL_2$ & $\GL_4\times \GL_{1}^{2n-2}\times \GL_2$ & $\GL_2\times \GL_2\times \GL_{1}^{2n-2}$ \\
        \hline
         {16} & $\GL_{2n+1}\times \GL_2$ & $\GL_{2n-1}\times \GL_2$ & $\GL_5\times \GL_{1}^{2n-4}\times \GL_2$ & $\GL_3\times \GL_2\times \GL_{1}^{2n-4}$  \\
        \hline
         {17} & $\Sp_{2p+2}\times \Sp_{2}$ & $\Sp_2\times \Sp_{2p}$ & $\Sp_4\times \Sp_2\times \GL_{1}^{p-1}$  & $\Sp_{2}^{2}\times \GL_{1}^{p-1}$ \\
        \hline
         {18} & $\Sp_{2p+2}\times \Sp_{2q+2}$ & $\Sp_2\times \Sp_{2p}\times \Sp_{2q}$ & $\Sp_4\times \Sp_4\times \GL_{1}^{p+q-2}$  & $\Sp_{2}^{3}\times \GL_{1}^{p+q-2}$ \\
        \hline
        19 & $\Sp_{2p+4}\times \Sp_4$ & $\Sp_{2p}\times \Sp_4$ & $\Sp_8\times \GL_{1}^{p-2}\times \Sp_4$ & $\Sp_{4}^{2}\times \GL_{1}^{p-2}$ \\
        \hline
         {20} & $\GL_{p+2}\times \Sp_{2q+2},\;p\leq 3$ & $\GL_p\times \SL_2\times \Sp_{2q}$ & $\GL_{p+2}\times \Sp_{4}\times \GL_{1}^{q-1}$ & $\GL_p\times \SL_2\times \SL_2\times \GL_{1}^{q-1}$ \\
        \hline
         {21} & $\GL_{2p+2}\times \Sp_{2q+2}$ & $\GL_{2p}\times \SL_2\times \Sp_{2q}$ & $\GL_{4}\times \Sp_{4}\times \GL_{1}^{2p+q-3}$ & $\GL_2\times \SL_2\times \SL_2\times \GL_{1}^{2p+q-3}$ \\
        \hline
         {22} & $\GL_{2p+3}\times \Sp_{2q+2}$ & $\GL_{2p+1}\times \SL_2\times \Sp_{2q}$ & $\GL_{5}\times \Sp_{4}\times \GL_{1}^{2p+q-3}$ & $\GL_3\times \SL_2\times \SL_2\times \GL_{1}^{2p+q-3}$ \\
        \hline
         {23} & $\Sp_4\times \Sp_{2p+2}\times \Sp_{2}$ & $\Sp_{2}^2\times \Sp_{2p}$  & $\Sp_{4}^{2}\times \Sp_{2}\times \GL_{1}^{p-1}$ & $ \Sp_{2}^{3}\times \GL_{1}^{p-1}$ \\
        \hline
         {24} & $\Sp_4\times \Sp_{2p+2}\times \Sp_{2q+2}$ & $\Sp_{2}^2\times \Sp_{2p}\times \Sp_{2q}$  & $\Sp_{4}^{3}\times \GL_{1}^{p+q-2}$ & $\Sp_{2}^{4}\times \GL_{1}^{p+q-2}$ \\
        \hline
         {25} & $\Sp_{2p+2}\times \Sp_{2}^{2}$ & $\Sp_{2p}\times \Sp_2$ & $\Sp_{4}\times \Sp_{2}^{2}\times \GL_{1}^{p-1}$ & $\Sp_{2}^{2}\times \GL_{1}^{p-1}$ \\
        \hline
         {26} & $\Sp_{2p+2}\times \Sp_{2q+2}\times \Sp_{2}$ & $\Sp_{2p}\times \Sp_{2q}\times \Sp_2$ & $\Sp_{4}^{2}\times \Sp_{2}\times \GL_{1}^{p+q-2}$ & $\Sp_{2}^{3}\times \GL_{1}^{p+q-2}$ \\
        \hline
         {27} & $\Sp_{2p+2}\times \Sp_{2q+2}\times \Sp_{2r+2}$ & $\Sp_{2p}\times \Sp_{2q}\times \Sp_{2r}\times \Sp_2$ & $\Sp_{4}^{3}\times \GL_{1}^{p+q+r-3}$ & $\Sp_{2}^{4}\times \GL_{1}^{p+q+r-3}$ \\
        \hline
	\end{tabular}
	\captionof{table}{Non-tempered spherical varieties for Type $A,B,C,D$ and $G$}
    \label{Table 1}
\end{figure}

\begin{figure}[h!]
	\begin{tabular}{| c | c | c |c| c| }
		\hline
		\textnumero & $G$ & $H$ & $L$ & $H_L$ \\
		\hline
		1 & $F_4$ & $\Spin_9$ & $\GL_2\times \GL_1\times \GL_1$ & $\GL_{1}^{4}$ \\
        \hline
        2 & $E_6$ & $F_4$ & $\GL_3\times \GL_3$ & $\GL_3\times \GL_1$ \\
        \hline
        3 & $E_6$ & $D_5$ & $\GL_4\times \GL_1\times \GL_1$ & $\GL_2\times \GL_2$ \\
        \hline
        4 & $E_7$ & $E_6$ & $\GL_3\times \GL_3\times \GL_1$ & $\GL_3$ \\
        \hline
        5 & $E_7$ & $A_1\times D_6$ & $E_6$ & $A_1\times A_5$ \\
        \hline
        6 & $E_8$ & $A_1\times E_7$ & $E_6$ & $A_1\times A_5$ \\
        \hline
	\end{tabular}
	\captionof{table}{Non-tempered spherical varieties for Type $E$ and $F$}
    \label{Table 2}
\end{figure}

Let $Q=LU$ be a parabolic subgroup of $G$ and $N_L$ be a maximal unipotent subgroup of $L$. Then $N=N_LU$ is a maximal unipotent subgroup of $G$. Let $\xi_L$ be a generic character of $[N_L]$ and we can extend it to a character of $[N]$ by making it trivial on $U$. Let $f$ (resp. $f'$) be a Schwartz function on $G(\BA)$ and $K_f$ (resp. $K_{f'}$) be the usual kernel function. In this paper, we want to compare the following two relative relative formulas.
$$I(f)=\int_{[N]}\int_{[H]}K_f(h,n)\xi_L(n)dhdn,\; J(f')=\int_{[N_L]}\int_{[H_L]}K_{f'}(h,n)\xi_L(n)dhdn.$$
If $(G,H)$ is tempered, $(L,H_L)=(G,H)$ and the above two relative trace formulas are identical. When $(G,H)$ is not tempered, such a comparison would relate the period integral of a non-tempered spherical variety to the period integral of a tempered spherical variety associated to a Levi subgroup. When $(G,H)=(\GL_{2n},\Sp_{2n}),\;(G_2,\SL_3),\;(\SO_7,G_2),\;(\SO_8,G_2)$ or $(\Sp_{4n},\Sp_{2n}\times \Sp_{2n})$, such a comparison has already been studied in some previous papers of Jacquet, Rallis, Jiang, Mao \cite{JR}, \cite{J}, \cite{JMR}, \cite{MR}. In this paper, we will generalized it to all the spherical varieties.

In general the comparison of two relative trace formulas are very difficult to prove, especially the fundamental lemma and smooth transfer. Also the comparison can not be given explicitly in the sense that for a Schwartz function $f$ on $G(\BA)$, we only know the existence of a function $f'\in J(f')$ such that $I(f)=J(f')$ instead of an explicit construction of $f'$. However, in the comparisons considered in this paper, there is an explicit formula for $f'$ in terms of $f$, which will be defined below.

\begin{defn}
Let $K_H\subset H(\BA)$ be a maximal compact group and $\chi$ be a character of $[L]$. For a Schwartz function $f$ of $G(\BA)$, we define a function $\CF_\chi(f)$ on $L(\BA)$ to be
$$\CF_\chi(f)(l)=\chi(l)\int_{U(\BA)}\int_{K_H}f(klu)dkdu.$$
\end{defn}

We also need a definition of local $L$-supercuspidal.

\begin{defn}
Let $F$ be a non-archimedean local field. We say a function $f\in C_{c}^{\infty}(G(F))$ is $L$-supercuspidal if it is a linear combination of the functions of the form 
$$\phi(g)= \begin{cases}
\phi_L(l) & \text{if } x = kl\in K L(F), \\
0 & \text{otherwise}.
\end{cases}$$
where $K$ is an open compact subgroup of $G(F)$ and $\phi_L$ is a matrix coefficient of a supercuspidal representation of $L(F)$. We say a Schwartz function $f$ on $G(\BA)$ is locally $L$-supercuspidal if it is of the form $f=\sum_{i=1}^{n}\phi_i$ where $\phi_i=\otimes_{v\in |k|}\phi_{i,v}\in \CS(G(\BA))$ such that for each $i$ the function $\phi_{i,v}$ is $L$-supercuspidal for some $v\in |k|$.
\end{defn}

Now we are ready to state our main result. Let $(G,H)$ and $(L,H_L)$ be as above. As $H_L$ is a Levi subgroup of $H$, we let $H_Q$ be a parabolic subgroup of $H$ with Levi factor $H_L$. It is clear from the table that the modular character $\delta_{H_Q}$ on $H_L$ can be extended to a character $\chi$ of $L$. Such an extension is not unique but we will just fix one. Then we can state our main theorem.

\begin{thm}\label{main theorem}
With the notation above, for all the models in Table \ref{Table 1}, we have
$$I(f)=J(\CF_\chi(f')).$$
for all $f\in \CS(G(\BA))$ that is locally $L$-supercuspidal.
\end{thm}

For the rest of this subsection, we will explain the proof of the theorem as well as why we need the locally $L$-supercuspidal condition (and when we do not need it). Let $G(k)=\cup_i Q(k)\gamma_i H(k)$ be the double coset decomposition (there maybe infinitely many orbits). For each $i$, we let $Q_i=Q\cap \gamma_i H\gamma_{i}^{-1}$, $H_i=H\cap \gamma_{i}^{-1}Q\gamma_i$ and $L_i$ be the projection of $Q_i$ to $L$. In Section 2, by using some unfolding technique, we will relate the relative trace $I(f)$ to the summation of the relative trace formula of $L_i\back L/ (N_L,\xi_L)$. Then in Section 3, we will prove the following proposition for all the models in Table \ref{Table 1} by a case-by-case argument.

\begin{prop}\label{main prop}
With the notation above, for all the models in Table \ref{Table 1}, the following hold.
\begin{enumerate}
\item There exists $i=i_0$ such that $L_{i_0}=H_L$ and $H_i$ is a parabolic subgroup of $H$ with Levi factor $H\cap \gamma_{i}^{-1}L\gamma_i$ and unipotent radical $H\cap \gamma_{i}^{-1}U\gamma_i$ (without loss of generality we may assume $\gamma_{i_0}=1$). 
\item For all $i\neq i_0$, the spherical varieity $L/L_i$ is either of parabolic induced type, or not tempered. Here we say a spherical subgroup of $L$ is of parabolic induced type if it contains the unipotent radical of a proper parabolic subgroup of $L$.
\end{enumerate}
\end{prop}

In the next subsection, we will give a conceptional explanation of the first part of the proposition from the point of view of the relative Langlands duality. The second part of the proposition remains mysterious as we do not have enough understanding of the non-open Borel/parabolic orbits of spherical varieties at this moment.

For an orbit $Q\gamma_i H$, if $i=i_0$, by the first part of the proposition, we know that the relative trace formula of $L_i\back L/ (N_L,\xi_L)$ just gives us $J(\CF_\chi(f))$. If $i\neq i_0$, there are two cases. If $L_i$ is not tempered, then by a result of Sakellaridis in Theorem 3 of \cite{P}, we know that the relative trace formula of $L_i\back L/ (N_L,\xi_L)$ vanishes for all Schwartz functions. If $L_i$ is of parabolic induced type, by the locally $L$-supercuspidal assumption we can show that the relative trace formula of $L_i\back L/ (N_L,\xi_L)$ vanishes. This completes the proof of the theorem.

\begin{rmk}
As we explained above, the only place we use the locally $L$-supercuspidal condition is when $L_i$ is of parabolic induced type. In other words, if $L_i$ is not tempered for all $i\neq i_0$, we can remove the locally $L$-supercuspidal condition in Theorem \ref{main theorem}. By our computations in Section 3, this applies to Model 3, Model 4 when $k=0$, Model 5-10, 17, 18, 20, 23-27 of Table \ref{Table 1}.  For those model, we actually proved that 
$I(f)=J(\CF_\chi(f'))$
for all $f\in \CS(G(\BA))$.
\end{rmk}

\begin{rmk}
The reason we only consider models in Table \ref{Table 1} in this paper is because the computation (i.e. the proof of Proposition \ref{main prop}) would be more complicated in Type $E$ and $F$. We will consider those models in a future paper.
\end{rmk}

\begin{rmk}
It is also possible to consider the comparison in the non-split cases, as long as the Levi subgroup $L$ of $G$ is defined over $k$. The computation will be similar to the split case, one only need to prove an analogue of Proposition \ref{main prop}. For example, one can study the comparison between the pair $U_{2n}/\Sp_{2n}$ and the pair $Res_{E/F}\GL_n/\GL_n$ where $E/F$ is the quadratic extension defining the quasi-split unitary group $U_{2n}$.
\end{rmk}

\begin{rmk}
Locally over a p-adic field where the character $\chi$ is unramified, the same argument as in Section 2.3 of \cite{JMR} shows that the map $\mapsto F_{\chi}(f)$ is a homomorphism from the Hecke algebra of $G$ to the Hecke algebra of $L$. Moreover, one can obtain an explicit description of this homomorphism similar to Proposition 2.4 of \cite{JMR}. In many cases, one can also study the spectral side of this relative trace formula comparison using a similar argument to that in Section 3 of \cite{JMR}. Finally, for many models in Tables \ref{Table 1} and \ref{Table 2}, one can apply the residue method to establish a relation between period integrals of residue representations for $G/H$ and period integrals of cuspidal automorphic representations for $L/H_L$. Combined with the Langlands-Shahidi method for residue representations, this yields relations between the period integrals of $L/H_L$ and certain automorphic L-functions. We refer the reader to \cite{PWZ2} for a discussion of the residue method and examples (which include several models from Tables \ref{Table 1} and \ref{Table 2}). In this paper, we do not pursue these directions further and instead focus on the relative trace formula identity in Theorem \ref{main theorem}.
\end{rmk}

\subsection{The relation with the relative Langlands duality and a general conjecture for Hamiltonian spaces}
In this subsection, we will give a conceptual explanation of our main theorem (Theorem \ref{main theorem}) from the point of view of the relative Langlands duality. We first recall the period integral conjecture in the relative Langlands duality. 

Let $G$ be a split connected reductive group and $\hat{G}$ be its dual group. Following Section 3.5 of \cite{BSV}, we say a smooth affine $G$-Hamiltonian space $\CM$ is hyperspherical if it satisfies the following three conditions:
\begin{itemize}
\item (coisotropic condition) The field of $G$-invariant rational functions on $\CM$ is commutative with respect to the Poisson bracket.
\item The image of the moment map $\CM\rightarrow \Fg^\ast$ has nonempty intersection with the nilcone of $\Fg^\ast$.
\item The stabilizer (in $G$) of a generic point of $\CM$ is connected.
\end{itemize}

In Section 3.6 of \cite{BSV}, Ben-Zvi, Sakellaridis, and Venkatesh proved a structure theorem for those Hamiltonian space which we will recall here. We define a BZSV quadruple for $G$ to be $\Delta=(G,H,\iota,\rho_H)$ where $H$ is a split reductive subgroup of $G$;  $\rho_H$  is a symplectic representation of $H$; and  $\iota$ is a homomorphism from $\SL_2$ into $G$ whose image commutes with $H$. For a BZSV quadruple $\Delta=(G,H,\iota,\rho_H)$ of $G$, by \cite[Section 3]{BSV}, one can associate a $G$-Hamiltonian variety $\CM_\Delta$. Theorem 3.6.1 of \cite{BSV} states that over an algebraic closed field any smooth affine hyperspherical $G$-Hamiltonian space is of the form $\CM_\Delta$ associated with a unique BZSV quadruple $\Delta=(G,H,\iota,\rho_H)$ of $G$. We say the quadruple $\Delta$ is hyperspherical if the associated Hamiltonian space $\CM_\Delta$ is hyperspherical.

\begin{rmk}
In the special case when $\iota$ is trivial and $\rho_H=0$ (i.e. $\Delta=(G,H,1,0)$), the associated Hamiltonian space $\CM_\Delta$ is just the cotangent bundle of the variety $G/H$. In this case, $\Delta$ is hyperspherical if and only if the variety $X=G/H$ is spherical and does not have Type N root (the coisotropic condition is equivalent to the spherical variety case in this case). 
\end{rmk}

For a BZSV quadruple $\Delta=(G,H,\iota,\rho_H)$, let $L$ the centralizer of $h(t):=\iota(\begin{pmatrix}t&0\\ 0&t^{-1}\end{pmatrix})$ in $G$  and by $U=\exp(\mathfrak u)$ (resp. $\bar{U}=\exp(\bar{\mathfrak u})$) the corresponding unipotent subgroups of $G$ associated with $\iota$, where $\mathfrak u\subset \Fg$ (resp. $\bar{\mathfrak u}\subset \Fg$) is the positive weight space (resp. negative weight space) of the Lie algebra $\Fg$ of $G$ under the adjoint action of $h(t)$. Then $P=LU$ and $\bar{P}=L\bar{U}$ are parabolic subgroups of $G$ that are opposite to each other. Since $H$ commutes with the image of $\iota$, we have $H\subset L$. 

Let $\mathfrak u^+$ be the $\geq 2$ weight space under the adjoint action of $h(t)$. It is well known that the vector space $\mathfrak u/\mathfrak u^+$ has a symplectic structure and realizes a symplectic representation of $H$ (and of $L$) under the adjoint action. We use $\rho_{\iota}$ to denote the symplectic representation $\Fu/\Fu^+$ of $H$ and let $\rho_{H,\iota}=\rho_H\oplus \rho_{\iota}$. 

\begin{defn}\label{defn anomaly-free}
We say a BZSV quadruple $\Delta=(G,H,\iota,\rho_H)$ is anomaly-free if the symplectic representation $\rho_{H,\iota}$ is an anomaly-free symplectic representation of $H$ (defined in Definition 5.1.2 of \cite{BSV}). We say a smooth affine hyperspherical Hamiltonian space $\CM=\CM_{\Delta}$ is anomaly-free if $\Delta$ is anomaly-free.    
\end{defn}

\begin{defn}\label{defn Whittaker induction}
For a BZSV quadruple $\Delta=(G,H,\iota,\rho_H)$, we define $\Delta_{red}=(L,H,1,\rho_{H,\iota})$ where $1$ stands for the trivial homomorphism from $\SL_2$ into $L$ (i.e. it maps every element to the identity). We say $\Delta$ is reductive if $\iota=1$ ($\iff \Delta=\Delta_{red}$).
\end{defn}

In \cite{BSV}, Ben-Zvi--Sakellaridis--Venkatehs proposed a conjectural duality between the set of smooth affine anomaly-free hyperspherical $G$-Hamiltonian spaces and the set of  smooth affine anomaly-free hyperspherical $\hat{G}$-Hamiltonian spaces, or equivalently, a conjectural duality between the set of anomaly-free hyperspherical BZSV quadruples of $G$ and the set of anomaly-free hyperspherical BZSV quadruples of $\hat{G}$. This proposed duality not only extends the classical Langlands program to a broader geometric setting but also provides a new perspective on the interaction between Hamiltonian symmetries and representation theory. They also formulated a series of elegant and far-reaching conjectures that should hold within this framework. 

\begin{rmk}
Despite its conceptual beauty, a major challenge in BZSV duality is the lack of a general algorithm to explicitly compute the dual of a given anomaly-free hyperspherical BZSV quadruples.  In other words, for a given anomaly-free hyperspherical BZSV quadruple $\Delta=(G,H,\iota,\rho_H)$, there is currently no known systematic procedure to determine its dual $\hat{\Delta}$. In Section 4 of \cite{BSV}, Ben-Zvi--Sakellaridis--Venkatehs devised an algorithm to compute the dual in a special case known as the polarized case, which is when the symplectic representation $\rho_{H,\iota}$ of $H$ is of the form $\rho_{H,\iota}=\tau\oplus \tau^\vee$ for some representation $\tau$ of $H$. In particular this include the cases when $\Delta=(G,H,1,0)$ (i.e. the spherical variety case). In a paper of Mao, Zhang and the author \cite{MWZ2}, they give an algorithm to compute the dual in the vector space case, i.e. the case when $\Delta=(G,G,1,\rho)$. In another ongoing paper of Tang, Zhang and the author, they give an algorithm to compute the dual when $G$ is a simple reductive group.
\end{rmk}

An important aspect of their conjecture concerns period integrals, which we will briefly recall here. Let $\Delta=(G,H,\iota,\rho_H)$ and $\hat{\Delta}=(\hat{G},\hat{H}',\hat{\iota}',\rho_{\hat{H}'})$ be two anomaly-free hyperspherical BZSV quadruples that are dual to each other. As we explained above, the map $\hat{\iota}'$ induce adjoint actions of \ $\hat{H}'\times \SL_2$ on $\hat{\Fg}$ and it can be decomposed as 
$$\hat{\Fg}=\oplus_{k\in \hat{I}} \hat{\rho}_k\otimes Sym^k$$
where $\hat{\rho}_k$ are representations of  $\hat{H}'$. It is clear that the adjoint representation of $\hat{H}'$ is a subrepresentation of $\hat{\rho}_0$. For an automorphic forms $\phi$ on $G(\BA)$, we can define the period integral $\CP_\Delta(\phi)$ associated to the quadruple $\Delta$ as in Chapter 1 of \cite{MWZ1} (when $\Delta=(G,H,1,0)$, $\CP_\Delta$ is just the one defined in \eqref{defn period integral}). The following conjecture is the main conjecture regarding period integrals in BZSV duality.

\begin{conj}\label{BSV conj} (Ben-Zvi--Sakellaridis--Venkatesh, Conjecture 14.3.5 and Equation (14.26) of \cite{BSV})
Let $\pi$ be an irreducible discrete automorphic representation of $G(\BA)$. For any embedding $\nu:\pi\rightarrow L^2(G(k)\back G(\BA))$, the period integral 
		$$\CP_{\Delta}(\phi),\;\phi\in Im(\nu)$$
		is nonzero only if the Arthur parameter of $\pi$ factors through $\hat{\iota}':\hat{H}'(\BC)\times \SL_2(\BC)\rightarrow \hat{G}(\BC)$. If this is the case, $\pi$ is a lifting of an Arthur packet $\Pi$ of $H'(\BA)$ (the Langlands dual group of $\hat{H}'$). Assume that $\Pi$ is tempered. Then we can choose the embedding $\nu$ so that 
		$$\frac{|\CP_{\Delta}(\phi)|^2}{\pair{\phi,\phi}}``=" \frac{L(1/2,\Pi,\rho_{\hat{H}'})\cdot\prod_{k\in \hat{I}}L(k/2+1,\Pi,\hat{\rho}_k)}{L(1,\Pi,Ad)^2}.$$
		Here $\pair{,}$ is the $L^2$-norm.
\end{conj}

\begin{rmk}\label{rmk BSV conj}
Roughly speaking, the conjecture asserts that the period integral associated to $\Delta$ equals the $L$-function associated to $\hat{\Delta}$. Conversely, if we interchange $\Delta$ and $\hat{\Delta}$, we similarly expect the period integral of $\hat{\Delta}$ to match the $L$-function associated to $\Delta$.

This conjecture is commonly known as an Ichino--Ikeda type conjecture. To state an explicit identity instead of using the notation ``$=$'', one must choose suitable Haar measures on $G$ and $H$, and make two adjustments to the right-hand side. We refer the reader to Remark 1.3 of \cite{MWZ1} for details.
\end{rmk}

\begin{rmk}\label{rmk Arthur}
Here we say an Arthur parameter $\phi:L_k\times \SL_2(\BC)\rightarrow \hat{G}(\BC)$ for a split reductive group $G$ is tempered if $\phi|_{\SL_2}=1$, where $L_k$ is the hypothetical Langlands group of $k$. 
If this is the case, we say the associated global Arthur packet is tempered. 

For a homomorphism $\hat{\iota}':\hat{H}'(\BC)\times \SL_2(\BC)\rightarrow \hat{G}(\BC)$ as in the above conjecture, we say an Arthur parameter $\phi:L_k\times \SL_2(\BC)\rightarrow \hat{G}(\BC)$ factors through $\hat{\iota}'$ if (up to conjugating $\phi$ by an element of $\hat{G}(\BC)$) there exists an Arthur parameter $\phi_{H'}:L_k\times \SL_2(\BC)\rightarrow \hat{H}'(\BC)$ of $H'$ such that $\phi_{H'}|_{L_k}=\phi|_{L_k}$ and $\phi(x)=\phi_{H'}(x)\hat{\iota}'(x)$ for $x\in \SL_2(\BC)$. In this case, we will say that the Arthur packet of $G(\BA)$ associated to $\phi$ is a lifting of the Arthur packet of $H'(\BA)$ associated to $\phi_{H'}$.
\end{rmk}

We say the quadruple $\Delta$ is tempered (resp. strongly tempered) if the $\SL_2$-homomorphism $\hat{\iota}'$ in its dual quadruple is trivial (resp. if $\hat{\iota}'$ is trivial and $\hat{H}'Z_{\hat{G}}=\hat{G}$). When $\Delta=(G,H,1,0)$ (i.e. the spherical variety case), $\Delta$ is tempered if and only if the spherical variety $X=G/H$ is tempered. By Conjecture \ref{BSV conj} above, it is expected that when the quadruple $\Delta$ is not tempered, the period integral $\CP_\Delta(\phi)$ should be equal to zero for all generic automorphic forms. Hence one needs to study the period for residue representations. 

Following Definition \ref{defn Whittaker induction}, we let $\hat{\Delta}_{red}=(\hat{L},\hat{H}',1,\hat{\rho}_{\hat{H}',\hat{\iota}'})$ with
$$\hat{\rho}_{\hat{H}',\hat{\iota}'}=\hat{\rho}_{\hat{H}'}\oplus \oplus_{k\in \hat{I},\; k\; \text{odd}}\hat{\rho}_k.$$
As $\hat{H}'$ is contained in the Levi subgroup $\hat{L}$, by Conjecture \ref{BSV conj}, if $\CP_{\Delta}(\phi)\neq 0$ for some automorphic forms $\phi$ of $G(\BA)$, the cuspidal support of $\phi$ must be contained in the Levi subgroup $L$ (here $L$ is the dual group of $\hat{L}$). Let $\pi_L$ be a tempered automorphic representation of $L(\BA)$ (i.e. its Arthur parameter is trivial on the $\SL_2$-component). We can apply Conjecture \ref{BSV conj} to the quadruple $\widehat{(\hat{\Delta}_{red})}$ and $\pi_L$. We get the following conjecture.

\begin{conj}\label{BSV conj for L}
For any embedding $\nu:\pi_L\rightarrow L^2(L(k)\back L(\BA))$, the period integral 
$$\CP_{\widehat{(\hat{\Delta}_{red})}}(\phi),\;\phi\in Im(\nu)$$
is nonzero only if the Arthur parameter of $\pi_L$ factors through $\hat{H}'(\BC)\subset \hat{L}(\BC)$. If this is the case, $\pi_L$ is a lifting of a tempered Arthur packet $\Pi$ of $H'(\BA)$ and we can choose the embedding $\nu$ so that 
$$\frac{|\CP_{\widehat{(\hat{\Delta}_{red})}}(\phi)|^2}{\pair{\phi,\phi}}``=" \frac{L(1/2,\Pi,\rho_{\hat{H}'})\cdot\prod_{k\in \hat{I},\;k\;\text{odd}}L(1/2,\Pi,\hat{\rho}_k)\cdot L(1,\Pi,\hat{\rho}_0)}{L(1,\Pi,Ad)^2}.$$
\end{conj}

Combining Conjecture \ref{BSV conj} and \ref{BSV conj for L}, it is clear that the period integrals of $\CP_{\Delta}$ and $\CP_{\widehat{(\hat{\Delta}_{red})}}$ are closely related to each other. First, they are not equal to zero only if the automorphic representation is a lifting of an automorphic representation $\Pi$ on $H'(\BA)$. And if this is the case, the associated L-values are also closely related to each other (we expect the difference between the L-values in Conjecture \ref{BSV conj} and \ref{BSV conj for L} to be related to the difference of the $L^2$-norms on $G(\BA)$ and $L(\BA)$). Meanwhile, the lifting from $H'(\BA)$ to $G(\BA)$ can be decomposed into two steps: one can first lift $\Pi$ from $H'(\BA)$ to $L(\BA)$, then build the Eisenstein series to $G(\BA)$ and take the residue. Moreover, if we assume that the lifting of $\Pi$ to $L(\BA)$ is tempered and cuspidal, then the constant term of the lifting of $\Pi$ to $G(\BA)$ along $U(\BA)$ (here $Q=LU$ is a parabolic subgroup of $G$) is essentially equal to its lifting to $L(\BA)$ (together with certain intertwining operators).\footnote{if we do not assume the lifting to $L(\BA)$ is cuspidal then the constant term is more complicated and this is one of the reasons why we make the $L$-supercuspidal assumption in the relative trace formula comparison.} In a way we can view $\widehat{(\hat{\Delta}_{red})}$ as the ``cuspidal support" of the quadruple $\Delta$.

%Moreover, if it is a lifting of a tempered automorphic representation $\Pi$ of $H'(\BA)$, then we should have
%\begin{equation}\label{equation 1}
%\frac{|\CP_{\Delta}(\phi_G)|^2}{\pair{\phi_G,\phi_G}}``=" \frac{L(1/2,\Pi,\rho_{\hat{H}'})\cdot\prod_{k\in \hat{I}}L(k/2+1,\Pi,\hat{\rho}_k)}{L(1,\Pi,Ad)^2},
%\end{equation}
%\begin{equation}\label{equation 2}
%\frac{|\CP_{\widehat{(\hat{\Delta}_{red})}}(\phi_L)|^2}{\pair{\phi_L,\phi_L}}``=" \frac{L(1/2,\Pi,\rho_{\hat{H}'})\cdot\prod_{k\in \hat{I},\;k\;\text{odd}}L(1/2,\Pi,\hat{\rho}_k)\cdot L(1,\Pi,\hat{\rho}_0)}{L(1,\Pi,Ad)^2}.   
%\end{equation}
%Here we use $\phi_G$ (resp. $\phi_L$) to denote the automorphic form on $G(\BA)$ (resp. $L(\BA)$). 

%Lastly, we expect that the difference of the L-functions on the right hand side of \eqref{equation 1} and \eqref{equation 2}, which is equal to 
%$$\frac{\prod_{k\in \hat{I},k>0}L(k/2+1,\Pi,\hat{\rho}_k)}{\prod_{k\in \hat{I},\;k\;\text{odd}}L(1/2,\Pi,\hat{\rho}_k)},$$
%should be closely related to the difference of the $L^2$-norm of $\phi_L$ and $\phi_G$ \footnote{the numerator is just the difference between the adjoint L-functions of $\phi_G$ and $\phi_L$}. In a way, we can view $\widehat{(\hat{\Delta}_{red})}$ as the ``cuspidal support" of the quadruple $\Delta$.

Motivated by the above discussion, we will make a general relative trace formula comparison regarding the period integrals $\CP_{\Delta}$ and $\CP_{\widehat{(\hat{\Delta}_{red})}}$. Let $Q=LU$ be a parabolic subgroup, $N_L$ be a maximal unipotent subgroup of $L$, $N=N_LU$ and $\xi_L$ be a generic character of $[N_L]$ (which will be extended to a character of $[N]$ by making it trivial on $U$) as in the previous subsection. Let $f$ (resp. $f'$) be a Schwartz function on $G(\BA)$ (resp. $L(\BA)$) and $K_f$ (resp. $K_{f'}$) be the kernel function as before. The relative trace formula $I(f)$ is given by taking the $\Delta$-period on the first variable of $K_f$ and the $(N,\xi_L)$-period on the second variable:
$$I(f)=\int_{[N]}\CP_{\Delta}(K_f(\cdot,n))\xi_L(n)dn. $$
On the other hand, the relative trace formula $J(f')$ is given by taking the $\widehat{(\hat{\Delta}_{red})}$-period on the first variable of $K_{f'}$ and the $(N_L,\xi_L)$-period on the second variable:
$$J(f')=\int_{[N_L]}\CP_{\widehat{(\hat{\Delta}_{red})}}(K_{f'}(\cdot,n))\xi_L(n)dn. $$

\begin{conj}\label{main conj}
With the notation above, there should be a character $\chi$ of $L$ (depends on $\Delta$) such that
$$I(f)=J(\CF_\chi(f'))$$
for all $f\in \CS(G(\BA))$ that is locally $L$-supercuspidal.
\end{conj}

When $\Delta$ is tempered the above conjecture is trivial as $I(f)$ and $J(f')$ are identically the same. When it is not tempered, then the above conjecture compares its period integral to the period integral of a tempered quadruple associated to a Levi subgroup.

When the quadruple is a spherical variety (i.e. when $\iota=1$ and $\rho_{H}=0$), the above conjecture is just Theorem \ref{main theorem}. In this case, the quadruple $\widehat{(\hat{\Delta}_{red})}$ is equal to $(L,H_L,1,0)$ where $H_L$ is given as in Table \ref{Table 1} and \ref{Table 2}. This gives a conceptual explanation of $(L,H_L)$ for those spherical varieties in Table \ref{Table 1} and \ref{Table 2} that do not have Type N spherical root (we refer the reader to the last part of this section for a discussion of the case with Type N root).

When the quadruple is polarized, we believe it is possible to prove the comparison in the above conjecture by a similar unfolding argument as in the spherical variety case. For example, the computation in \cite{JQ} can be used to prove Conjecture \ref{main conj} for the quadruple $(\SO_{4n},\Sp_{2n},(2^{2n}),0)$ and the computation in \cite{PWZ} can be used to prove Conjecture \ref{main conj} for the quadruple $(E_6,G_2, A_2\times A_2, 0)$. Here $2^{2n}$ (resp. $A_2\times A_2$) denotes the nilpotent orbit of $\SO_{4n}$ (resp. $E_6$) that is principal in the Levi subgroup $\GL_{2}^{n}$ (resp. $\GL_3\times \GL_3$).

\begin{rmk}
In a previous paper by the author with Mao and Zhang \cite{MWZ1}, a conjectural relative trace formula comparison was made between any BZSV quadruple $\Delta$
 and a strongly tempered quadruple $\Delta'$. Although the conjectural comparison in \cite{MWZ1} is more powerful than the one in Conjecture \ref{main conj} as it reduces the study of period integrals to strongly tempered cases, it is very difficult to prove in general (so far it has only been established in some special lower-rank cases). On the other hand, the comparison we propose in Conjecture \ref{main conj}, while only reducing the study of period integrals to tempered cases, is relatively easier to prove, especially in the polarized case.
\end{rmk}

Next we will give a conceptual explanation of the first part of Proposition \ref{main prop}. To do this, we would need to recall a conjecture from \cite{BSV} about the relations between $\Delta=(G,H,\iota,\rho_H)$ and $\widehat{(\hat{\Delta}_{red})}$. We first need a definition.

\begin{defn}
Let $L$ be a Levi subgroup of $G$ and $\rho$ be an irreducible representation of $L$ with the highest weight $\varpi_L$. There exists a Weyl element $w$ of $G$ such that $w\varpi_L$ is a dominant weight of $G$ \footnote{the choice of $w$ is not unique but $w\varpi_L$ is uniquely determined by $\varpi_L$}. We define $(\rho)_{L}^{G}$ to be the irreducible representation of $G$ whose highest weight is $w\varpi_L$. In general, if $\rho=\oplus_i\rho_i$ is a finite-dimensional representation of $L$ with $\rho_i$ irreducible, we define
$$(\rho)_{L}^{G}=\oplus_i(\rho_i)_{L}^{G}.$$
\end{defn}
		
Now we are ready to state the conjecture. Let $\widehat{(\hat{\Delta}_{red})}=(L,H_L,\iota_L,\rho_{H_L})$.

\begin{conj}(Ben-Zvi--Sakellaridis--Venkatesh, Section 4.2.2 of \cite{BSV})\label{Whittaker induction dual}
With the notation above, the following holds.
\begin{enumerate}
\item $\iota=\iota_L$.
\item $H_L$ is a Levi subgroup of $H$ and $H$ is generated by $H_L$ and $\{Im(\iota_\alpha)\}$ where $\alpha$ runs over simple roots of $G$ that does not belong to $L$ and $\iota_\alpha$ is the $\SL_2$-embedding that maps the simple root of $\SL_2$ into $\alpha$.
\item $\rho_H=(\rho_{H_L})_{H_L}^{H}$.
\end{enumerate}
\end{conj}

Now if we are in the spherical variety case (i.e. $\Delta=(G,H,1,0)$), the first and third parts of the above conjecture would imply that $\iota_L=1$ and $\rho_{H_L}=0$. In particular, the quadruple $\widehat{(\hat{\Delta}_{red})}$ is associated to a spherical variety $L/H_L$ of $L$. Moreover, the second part of the above conjecture implies that $H_L$ is a Levi subgroup of $H$ and $H$ is generated by $H_L$ and $\{Im(\iota_\alpha)\}$ where $\alpha$ runs over simple roots of $G$ that does not belong to $L$. Now if we let $Q=LU$ be the standard parabolic subgroup of $G$, then this would implies that $H\cap Q=(H\cap L)\ltimes (H\cap U)$ is a parabolic subgroup of $H$ and $H\cap L=H_L$. This gives a conceptual explanation of the first part of Proposition \ref{main prop}.

Lastly, we would like point out that the result we proved in Theorem \ref{main theorem} actually goes beyond the current setting of the relative Langlands duality. The reason is that in Theorem \ref{main theorem} we do not make the assumption that the spherical variety $X=G/H$ does not have Type N spherical root, while in the setting of the relative Langlands duality, in order for the quadruple $\Delta=(G,H,1,0)$ to be hyperspherical, the spherical variety $X=G/H$ can not have Type N spherical root. Among the models in Table \ref{Table 1} and \ref{Table 2}, Model 4, 9 of Table \ref{Table 1} and Model 5, 6 of Table \ref{Table 2} have Type N spherical root.

As a result, we believe the comparison in Conjecture \ref{main conj} should goes beyond anomaly-free hyperspherical Hamiltonian spaces. In fact, we believe we only need to assume that the Hamiltonian space is coisotropic. In the theory of relative Langlands duality, the generic stabilizer assumption in the definition of hyperspherical and the anomaly-free assumption ensure that there is no covering group involved in the duality \footnote{More precisely, if the generic stabilizer is not connected, one expects covering group appears in the dual quadruple. On the other hand, without the anomaly-free condition, the image of the symplectic representation $\rho_{H,\iota}$ does not split over the metaplectic cover and the period integral need to be taken over a covering group of $H$.}. However, we still expect that in this case one can define an analogue of the quadruple  $\widehat{(\hat{\Delta}_{red})}$ satisfying the conditions in Conjecture \ref{Whittaker induction dual} such that Conjecture \ref{main conj} holds. Actually we believe the Levi subgroup $L$ can be defined in the general coisotropic case (without the assumption on generic stabilizer and anomaly-free) using the generic stabilizer of $G$ acts on the Hamiltonian space. However it is not clear to us how to define the remaining datum in $\widehat{(\hat{\Delta}_{red})}$ for general coisotropic Hamiltonian spaces at this moment (except in the spherical variety case as in Table \ref{Table 1} and \ref{Table 2}). This is why we still assume anomaly-free and hyperspherical in Conjecture \ref{main conj}.

\subsection{Organization of the paper}
In Section 2, we will use the unfolding method to prove the main theorem by assuming Proposition \ref{main prop}. In Section 3, we will prove Proposition \ref{main prop} by a case-by-case computation argument.

\subsection{Acknowledgement} The author’s work is partially supported by the NSF grant DMS-2349836 and a Simons Travel Grant. 

\section{The proof of the relative trace formula identity}

In this section we will prove the main theorem (Theorem \ref{main theorem}) by assuming Proposition \ref{main prop}. Proposition \ref{main prop} will be proved in the next subsection. Recall that $G(k)=\cup_i Q(k)\gamma_i H(k)$ is the double coset decomposition (there maybe infinitely many orbits), $Q_i=Q\cap \gamma_i H\gamma_{i}^{-1}$, $H_i=H\cap \gamma_{i}^{-1}Q\gamma_i$ and $L_i$ is the projection of $Q_i$ to $L$.  We have
$$I(f)=\int_{[H]}\int_{[N]} K_f(h,n)\xi_L(n)dndh= \int_{[H]}\int_{[N_L]}\int_{[U]} K_f(h,un)\xi_L(n)dudndh$$
$$=\int_{[N_L]}\int_{[U]} \sum_{\delta\in H(k)\back G(k)} \int_{H(\BA)} f(h\delta un)  \xi_L(n)dhdudn$$
$$=\sum_i  \int_{[N_L]}\int_{[U]} \sum_{\delta\in H(k)\back H(k)\gamma_i Q(k)} \int_{H(\BA)} f(h\delta un)  \xi_L(n)dhdudn=\sum_{i}I_i(f).$$

\begin{prop}
When $i\neq i_0$, we have $I_i(f)=0$.
\end{prop}

\begin{proof}
We first consider the case when $L/L_i$ is of parabolic induced type. In this case the subgroup $L_i$ contains the unipotent radical $U_{L,i}$ of a proper parabolic subgroup of $L$. We have (here $U_i=U\cap \gamma_i H\gamma_{i}^{-1}$ and we will choose the Haar measure so that $vol(U_i(k)\back U_i(\BA))=1$)
$$I_i(f)=\int_{[N_L]}\int_{[U]} \sum_{\delta\in H(k)\back H(k)\gamma_i Q(k)} \int_{H(\BA)} f(h\delta un)  \xi_L(n)dhdudn$$
$$=  \int_{[N_L]} \sum_{\delta\in L_i(k)\back L(k)} \int_{U_i(k)\back U(\BA)} \int_{H(\BA)} f(h\gamma_i  u \delta n)  \xi_L(n)dhdudn.$$
$$=\int_{[N_L]} \sum_{\delta\in L_i(k)\back L(k)}\int_{U_i(\BA)\back U(\BA)}\int_{H(\BA)} f(h\gamma_i  u \delta n)  \xi_L(n)dhdudn.$$
Then it is enough to show that for any $l\in L(\BA)$, we have 
$$\int_{U_i(\BA)\back U(\BA)}\int_{H(\BA)} f_{v_0}(h\gamma_i  u l)  dhdu=0.$$
We only need to prove the vanishing of the above integral at a local place $v_0$. Hence it is enough to prove the following lemma.

\begin{lem}
Let $F$ be a p-adic local field and $f\in C_{c}^{\infty}(G(F))$. If $f$ is $L$-supercuspidal, then
$$\int_{U_i(F)\back U(F)} \int_{H(F)} f(h\gamma_i u)dhdu=0.$$
\end{lem}

\begin{proof}
Let $H'=H\cap \gamma_i U_{L,i} U\gamma_{i}^{-1}$. Then we can rewrite the integral as
$$\int_{H(F)/H'(F)} \int_{U(F)}\int_{U_{L,i}(F)} f(h\gamma_i uu')du' du dh.$$
The the vanishing follows from the integral over $\int_{U_{L,i}(F)}$.
\end{proof}

Next we consider the case when $L/L_i$ is not tempered. We have
$$I_i(f)= \int_{[N_L]} \sum_{\delta\in L_i(k)\back L(k)} \int_{U_i(k)\back U(\BA)} \int_{H(\BA)} f(h\gamma_i  u \delta n)  \xi_L(n)dhdudn$$
$$=\sum_{\delta\in L_i(k)\back L(k)/N_L(k)}  \int_{(N_L\cap \delta^{-1} L_i\delta)(k)\back N_L(\BA)} \int_{U_i(k)\back U(\BA)}\int_{H(\BA)} f(h\gamma_i u \delta n)  \xi_L(n)dhdudn.$$
By a result of Sakellaridis in Theorem 3 of \cite{P}, for any $\delta\in L$, the addtive character $\xi_L$ is nontrivial on $N_L\cap \delta^{-1} L_i\delta$. In particular the outer integral is zero for all $\delta$. This proves the proposition.
\end{proof}

Now it remains to study the term $I_{i_0}(f)$. For simplicity, we may assume that $\gamma_{i_0}=1$. By Proposition \ref{main prop}(1), we have $H_Q=Q\cap P=(H\cap L) \ltimes (H\cap U)=H_L\cap U_H$ is a parabolic subgroup of $H$. We have
$$I_{i_0}(f)=\int_{[N_L]}\int_{[U]} \sum_{\delta\in H(k)\back H(k) Q(k)} \int_{H(\BA)} f(h\delta un)  \xi_L(n)dhdudn$$
$$=\sum_{\delta\in H_L(k)\back L(k)} \int_{[N_L]}\int_{U_H(\BA)\back U(\BA)}\int_{H(\BA)} f(h u \delta n)  \xi_L(n)dhdudn $$
$$=\sum_{\delta\in H_L(k)\back L(k)} \int_{[N_L]} \int_{U_H(\BA)\back U(\BA)} \int_{U_{H}(\BA)} \int_{H_L(\BA)} \int_{K_H} f(klu' u \delta n)\xi_L(n)\delta_{H_Q}(l) dk dldu'dudn $$
$$=\sum_{\delta\in H_L(k)\back L(k)} \int_{[N_L]} \int_{H_L(\BA)} \int_{U(\BA)}\int_{K_H}   f(kl \delta n u)\xi_L(n)\delta_{H_Q}(l) dkdudldn$$
$$=\sum_{\delta\in H_L(k)\back L(k)} \int_{[N_L]}  F_{L,\chi}(f)(l \delta n)\xi_L(n) dl dn=I_L(F_{L,\chi}(f)).$$
Here we recall that 
$$\CF_\chi(f)(l)=\chi(l)\int_{U(\BA)}\int_{K_H}f(klu)dkdu$$
and $\chi$ is a character of $L$ whose restriction to $H_L$ is equal to $\delta_{H_Q}$. This finishes the proof of Theorem \ref{main theorem}.

\section{The proof of Proposition \ref{main prop}}
In this section we will prove Proposition \ref{main prop} for all the models in Table \ref{Table 1}.

For the first model, $(G,H)=(\GL_{2a+2k},\GL_a\times \GL_{2k})$ and $(L,H_L)=(\GL_{2a}\times \GL_{1}^{2k},\GL_{a}\times \GL_a\times \GL_{1}^{2k})$. We let $P_{p,q}$ (resp. $L_{p,q}$) be the standard upper triangular parabolic subgroup (resp. standard Levi subgroup) of $\GL_{p+q}$ of type $(p,q)$ and we first consider the double coset
$$P_{a,a+2k}\backslash \GL_{2a+2k}/P_{2a,2k}.$$
By the Bruhat decomposition this double coset has $min\{a,2k\}+1$ many orbits and the stabilizer of the orbit in $P_{a,a+2k}$ is of the form
$$\begin{pmatrix}\GL_{a-i} & \ast & \ast & \ast\\ 0 & \GL_i & 0 & \ast \\ 0 & 0 & \GL_{a+i} & \ast \\ 0 & 0 & 0 & \GL_{2k-i}  \end{pmatrix}$$
and the stablizer in $P_{2a,2k}$ is of the form
$$\begin{pmatrix}\GL_{a-i} & \ast & \ast & \ast\\ 0 & \GL_{a+i} & 0 & \ast \\ 0 & 0 & \GL_{i} & \ast \\ 0 & 0 & 0 & \GL_{2k-i}  \end{pmatrix}$$
with $0\leq i\leq min\{a,2k\}$. If we break into $L_{a,a+2k}\backslash \GL_{2a+2k}/P_{2a,2k}$ orbits, we just need to study the action of $\GL_i\times \GL_{a+i}$ on $Mat_{i\times (a+i)}$ and the orbits are given by the rank of the matrix. Namely, the orbits of $L_{a,a+2k}\backslash \GL_{2a+2k}/P_{2a,2k}$ are parametrized by $(i,j)$ with $i$ as before and $0\leq j\leq i$. For each $(i,j)$, the projection of the  stabilizer in $P_{2a,2k}$ to $L_{2a,2k}$ is given by 
$$diag(\begin{pmatrix}\GL_{a-i} & \ast & 0 \\ 0 & A & 0 \\ 0 & \ast & C \end{pmatrix},\begin{pmatrix} A & D & \ast \\ 0 & E & 0 \\ 0 & 0 & \GL_{2k-i}\end{pmatrix}),\; A\in \GL_j.$$

The last thing is to furthur decompose 
$$\GL_{2k}/\{\begin{pmatrix} A & D & \ast \\ 0 & E & 0 \\ 0 & 0 & \GL_{2k-i}\end{pmatrix})\}$$ 
into Borel orbits (recall that $Q$ is the parabolic subgroup of type $(2a,1^{2k})$). When $i=0$, there is only one orbit and it is clear that this orbit satisfies Proposition \ref{main prop}(1). When $i>0$, for each Borel orbit, the $A$-part of the stabilizer in $\{\begin{pmatrix} A & D & \ast \\ 0 & E & 0 \\ 0 & 0 & \GL_{2k-i}\end{pmatrix})\}$ is a Borel subgroup. Hence the projection of the stabilizer to $\GL_{2a}\subset L=\GL_{2a}\times \GL_{1}^{2k}$ is 
$$\begin{pmatrix}\GL_{a-i} & \ast & 0 \\ 0 & A & 0 \\ 0 & \ast & C \end{pmatrix},\;A\in B_j$$
where $B_j$ is a Borel subgroup of $\GL_j$. If $j>0$,
the stabilizer contains the unipotent radical $\begin{pmatrix}I_{a-i}& \ast & 0 \\ 0 & I_j & 0 \\ 0 & \ast & I_{a+i-j} \end{pmatrix}$ of a proper parabolic subgroup of $\GL_{2a}$. If $j=0$, the stabilizer is $\GL_{a-i}\times \GL_{a+i}\subset \GL_{2a}$ which is not tempered when $i>0$. This proves Proposition \ref{main prop}.

For the second model, $(G,H)=(\GL_{2a+2k+1},\GL_a\times \GL_{2k+1})$ and $(L,H_L)=(\GL_{2a+1}\times \GL_{1}^{2k},\GL_{a+1}\times \GL_a\times \GL_{1}^{2k})$. The calculation is very similar to the first model and we will skip it here.

For the third model, $(G,H)=(\GL_{2n},\Sp_{2n})$ and $(L,H_L)=(\GL_n\times \GL_n,\GL_n)$. In this case, $G/Q$ corresponds to all the $n$-dimensional subspaces of a $2n$-dimensional vector space. It is clear that $H=\Sp_{2n}$-orbits on it are parametrized by the dimension of the maximal non-degenerate symplectic subspace of the $n$-dimensional subspace (which can be $2i$ with $0\leq i\leq [n/2]+1$, in particular there are $[n/2]+1$ many orbits). If $i=0$ (i.e. the $n$-dimensional subspace is isotropic), then it is clear that this orbit satisfies Proposition \ref{main prop}(1). If $i>0$, the stabilizer $L_i$ is given by
$$\begin{pmatrix}A &B_1\\ 0&C_1 \end{pmatrix}\times \begin{pmatrix}A^\ast &B_2\\ 0&C_2 \end{pmatrix},\; A\in GL_{n-2i}, A^\ast=(A^t)^{-1},\;B_i\in Mat_{(n-2i)\times 2i}, C_i\in Sp_{2i}$$
and is not tempered (this is because the spherical variety $\GL_{2i}/\Sp_{2i}$ is not tempered). This proves Proposition \ref{main prop}. In this case since $L_i$ is not tempered for $i\neq 0$, we can remove the locally L-supercuspidal condition in Theorem \ref{main theorem}.

For Model 4, $(G,H)=(\Sp_{4m+2k},\Sp_{2m}\times \Sp_{2m+2k})$ and $(L,H_L)=(\GL_{2m}\times \GL_{1}^{k},\GL_m\times \GL_m\times \GL_{1}^{k})$. We first consider the double coset $Q'\back G/H$ with respect to the parabolic $Q'$ containing $Q$ whose Levi factor is $\GL_{2m}\times \Sp_{2k}$. The double coset corresponds to the action of $\Sp_{2m}\times \Sp_{2m+2k}$ on $2m$-dimensional isotropic subspaces. When $k=0$, the calculation has been done in Lemma 2.1 of \cite{GPSS}. The calculation for the general case is similar.

Let $V=V_1\oplus V_2$ be the symplectic space defining $G$ with $V_1$ (resp. $V_2$) corresponds to the $2m$-dimensional (resp. $2m+2k$-dimensional) subspace defining the $\Sp_{2m}$-part (resp. $\Sp_{2m+2k}$-part) of $H$. Let $W$ be a $2m$-dimensional isotropic subspace, $k_i=\dim(W\cap V_i)$. Then $2m-k_2$ (resp. $2m-k_1$) are the dimension of the projection of $W$ on $V_i$ (we use $W_i$ to denote this space). Since $W$ is isotropic, the dimension of the maximal anisotropic subspace of $W_1$ and $W_2$ are equal to each other, we set this number to be $2k_3$. These three numbers need to satisfies some conditions. 

\begin{itemize}
\item It is clear that $0\leq k_1\leq m$, $0\leq k_2\leq m+k$, and $k_1+k_2\leq 2m$.
\item Since $W\cap V_i$ is orthogonal to $W_i$, its dimension must be less or equal to $W_{i}^{\perp}$ where $W_{i}^{\perp}$ is suspace of vectors in $V_i$ that are orthogonal to $W_i$. This implies that $2m-k_2\leq 2m-k_1$ and $2m-k_1\leq 2m+2k-k_2$. Hence we have
$$0\leq k_2-k_1\leq 2k.$$
\item Since $W\cap V_i$ is orthogonal to $W_i$ and $2k_3$ is the dimension of the maximal anisotropic subspace of $W_i$, we have $2k_3\leq \dim(W_i)-\dim(W\cap V_i)$ which implies that 
$$2k_3\leq 2m-k_1-k_2.$$
\item Since the maximal isotropic subspace of $W_i$ has dimension $\dim(W_i)-2k_3$, which needs to be less or equal to the dimension of the maximal isotropic subspace of $V_i$ minus $k_3$ (here $k_3$ is the dimension of the maximal isotropic subspace of the maximal anisotropic subspace of $W_i$). This implies that $2m-k_2-2k_3\leq m-k_3$ and $2m-k_1-2k_3\leq m+k-k_3$. Hence we have
$$k_3\geq m-k_2,\;k_3\geq m-k-k_1.$$

\end{itemize}

\begin{rmk}\label{sp remark}
In this special case when $k=0$, the second inequality implies that $k_1=k_2$, and the last two inequalities imply that $k_3=m-k_1$. Hence the only invariant is just $k_1$, which is exactly Lemma 2.1 of \cite{GPSS}.    
\end{rmk}

By an argument analogous to Lemma 2.1 in \cite{GPSS}, the orbits in $Q' \backslash G/H$ are parametrized by the integers $k_1, k_2, k_3$ satisfying the conditions above. Specifically, for each such triple, the $2m$-dimensional isotropic subspace $W \subset V$ decomposes as
$$W=W_1\oplus W_2\oplus W_{iso,diag}\oplus W_{aniso,diag}$$
Here, $W_i \subset V_i$ is an isotropic subspace of dimension $k_i$. The remaining components are defined as follows.
\begin{itemize}
\item $W_{\mathrm{iso,diag}}$ is a subspace of $W_{\mathrm{iso},1} \oplus W_{\mathrm{iso},2}$ that projects surjectively onto each $W_{\mathrm{iso},i}$, where $W_{\mathrm{iso},i} \subset V_i$ is an isotropic subspace of dimension $2m - k_1 - k_2 - 2k_3$.
\item $W_{\mathrm{aniso,diag}}$ is a subspace of $W_{\mathrm{aniso},1} \oplus W_{\mathrm{aniso},2}$ that projects surjectively onto each $W_{\mathrm{aniso},i}$, where $W_{\mathrm{aniso},i} \subset V_i$ is an anisotropic subspace of dimension $2k_3$.
\end{itemize}

As for the stabilizers, for each $(k_1,k_2,k_3)$, the projection of the stabilizer in $Q'$ to the Levi subgroup $\GL_{2m}\times \Sp_{2k}$ is of the form
\begin{equation}\label{sp comuptation}
\begin{pmatrix}\GL_{k_1} & 0 & \ast & \ast \\ 0 & \GL_{k_2} & \ast & \ast \\ 0 & 0 & \Sp_{2k_3} & \ast \\ 0 & 0 & 0 & \GL_{2m-k_1-k_2-2k_3}^{diag} \end{pmatrix}
\end{equation}
$$\times \begin{pmatrix}\GL_{2m-k_1-k_2-2k_3}^{diag} & \ast & \ast \\ 0 & diag(\Sp_{2k_2+2k_3-2m}\times \Sp_{2k_1+2k_3+2k-2m}) & \ast\\ 0 & 0 & (\GL_{2m-k_1-k_2-2k_3}^{diag})^\sharp \end{pmatrix},\;g^\sharp=(g^t)^{-1}.$$
When we future take the Borel orbits of the second component in $\Sp_{2k}$ to get the orbits for $Q\back G/H$, the projection of the stabilizer of each orbit to $\GL_{2m}$ becomes
$$\begin{pmatrix}\GL_{k_1} & 0 & \ast & \ast \\ 0 & \GL_{k_2} & \ast & \ast \\ 0 & 0 & \Sp_{2k_3} & \ast \\ 0 & 0 & 0 & B_{\GL_{2m-k_1-k_2-2k_3}}\end{pmatrix}.$$
When $2m-k_1-k_2-2k_3\neq 0$, the stabilizer contains the unipotent radical of a proper parabolic subgroup of $\GL_{2m}$. When $k_3\neq 0$ (resp. $|k_1-k_2|> 1$), the subgroup is not tempered since $\Sp_{2k_3}\subset \GL_{2k_3}$ (resp. $\GL_{k_1}\times \GL_{k_2}\subset \GL_{k_1+k_2}$) is not tempered. It remains to consider the case when $k_3=2m-k_1-k_2-2k_3=0$ and $|k_1-k_2|\leq 1$. In this case, since $k_1+k_2=2m$ we must have $k_1=k_2=m$. Moreover, in this case the second component of \eqref{sp comuptation} is just $\Sp_{2k}$ and hence it only breaks into one Borel orbit. This is exactly the orbit that satisfies Proposition \ref{main prop}(1). This proves Proposition \ref{main prop}.

\begin{rmk}
In the special case when $k=0$, as we explained in Remark \ref{sp remark}, $2m-k_1-k_2-2k_3$ must always be $0$ and hence $L_i$ is not tempered for all $i\neq i_0$. In particular we can remove the locally L-supercuspidal condition in Theorem \ref{main theorem}. This is the case considered by Mao and Rallis in \cite{MR}.
\end{rmk}

For Model 5, $(G,H)=(\Sp_{2n},\Sp_{2n-2}\times \GL_1)$ and $(L,H_L)=(\Sp_4\times \GL_{1}^{n-2},\Sp_2\times \GL_{1}^{n-1})$. In this case, let $W$ be the $2n$-dimensional symplectic space defining $G$. The double coset $Q\back G/H$ corresponds to the action of $Q$ on the set of two lines of $W$ that are not orthogonal to each other. We can decompose the symplectic space $W$ as $W_{n-2,+}\oplus W_4\oplus W_{n-2,-}$ where $V_4$ corresponds to the $\Sp_4$-part of $L$ and $W_{n-2,+},W_{n-2,-}$ are the isotropic subspaces of dimension $n-2$ that corresponds to the $\GL_{1}^{n-2}$-part of $L$ (in particular the parabolic subgroup $Q$ preserves the isotropic subspace $W_{n-2,+}$). If the projection of the two lines on $W_{n-2,-}$ is two-dimensional, then we can use the unipotent element in $Q$ to make it belong to $W_{n-2,+}\oplus W_{n-2,-}$ and hence the projection of the stabilizer to $\Sp_{4}$-part of $L$ is $\Sp_4$, which is not tempered. 

If the projection of the two lines on $W_{n-2,-}$ is one-dimensional, then we can use the unipotent element in $Q$ to make one line belong to $W_{n-2,+}\oplus W_{n-2,-}$. Then the projection of the stabilizer to $\Sp_{4}$-part of $L$ contains the mirabolic subgroup of $\Sp_4$, which is not tempered. 

If the projection of the two lines on $W_{n-2,-}$ is zero-dimensional, then its projection on $V_4$ is two-dimensional (because these two lines are not orthogonal to each other and $W_{n-2,+}$ is isotropic) and we can use unipotent element in $Q$ to make it in $W_4$. Since $\Sp_4$ acts transitively on the set of  two lines of $W_4$ that are not orthogonal to each other, those two lines (i.e. the set of two lines whose projection on $W_{n-2,-}$ is zero-dimensional) form one $Q$-orbit and it is clear that it satisfies Proposition \ref{main prop}(1). This proves Proposition \ref{main prop}. Moreover, in this case $L_i$ is not tempered for $i\neq i_0$ and therefore we can remove the locally L-supercuspidal condition in Theorem \ref{main theorem}.

For Model 6, $(G,H)=(\SO_{2k+2m+1},\SO_m\times \SO_{m+2k+1})$ and $(L,H_L)=(\SO_{2m+1}\times \GL_{1}^{k},\SO_{m}\times \SO_{m+1}\times \GL_{1}^{k})$. The argument will be similar to Model 4. We first consider the double coset for $Q'\back G/H$ with respect to the parabolic $Q'$ containing $Q$ whose Levi factor is $\SO_{2m+1}\times \GL_{k}$. This is similar to the symplectic case in Model 4, we can decomposition the quadratic space defining $G$ as $V=V_1\oplus V_2$ with $\dim(V_1)=m$ and $\dim(V_2)=m+2k+1$ and let $W$ be a $k$-dimensional isotropic subspace. Let $k_i=\dim(W\cap V_i)$. Then $k-k_2$ (resp. $k-k_1$) is the dimension of the projection of $W$ on $V_i$ (we use $W_i$ to denote this space). Since $W$ is isotropic, the dimension of the maximal anisotropic subspace of $W_1$ and $W_2$ are equal to each other, we set this number to be $k_3$. Similar to the symplectic case, these three numbers need to satisfy the following inequalities 
$$0\leq k_1\leq [m/2],\;0\leq k_2\leq [\frac{m+2k+1}{2}],\;k_1+k_2\leq k,\;k-m\leq k_2-k_1\leq m+k+1,\;k_3\leq k-k_1-k_2,$$
$$k_3\geq 2k-2k_2-m.$$
Moreover, unlike the symplectic case in which all the $2k_3$-dimensional symplectic subspaces are conjugated to each other, in the orthogonal case, the orbits are parametrized by the numbers $(k_1,k_2,k_3)$ together with a non-denegerate orthogonal subspace $V_0$ of $V_1$ of dimension $k_3$ (up to $\SO(V_1)-$conjugation). When $0<k_3<m$, there will be infinitely many such $V_0$ and hence there are infinitely many orbits. For each $(k_1,k_2,k_3)$ and $V_0$, the projection of the stabilizer in $Q'$ to the Levi subgroup $\GL_k\times \SO_{2m+1}$ is
\begin{equation}\label{so equation}
\begin{pmatrix}\GL_{k_1} & 0 & \ast & \ast \\ 0 & \GL_{k_2} & \ast & \ast \\ 0 & 0 & \SO(V_0) & \ast \\ 0 & 0 & 0 & \GL_{k-k_1-k_2-k_3}^{diag} \end{pmatrix}
\end{equation}
$$\times \begin{pmatrix}\GL_{k-k_1-k_2-k_3}^{diag} & \ast & \ast \\ 0 & diag(\SO_{m-2k+2k_2+k_3}\times \SO_{m+1+2k_1+k_3}) & \ast\\ 0 & 0 & (\GL_{k-k_1-k_2-k_3}^{diag})^\ast \end{pmatrix}.$$

When we futhur take Borel orbits of the first component in $\GL_k$ (note that when $k_3>2$ there are infinitely many Borel orbits) to get the orbits for $Q'\back G/H$, the projection of the stabilizer of each orbit to the $\SO_{2m+1}$-part of $L$ becomes 
$$\begin{pmatrix}B_{k-k_1-k_2-k_3}^{diag} & \ast & \ast \\ 0 & diag(\SO_{m-2k+2k_2+k_3}\times \SO_{m+1+2k_1+k_3}) & \ast\\ 0 & 0 & (B_{k-k_1-k_2-k_3}^{diag})^\ast \end{pmatrix}.$$

We have $(m+1+2k_1+k_3)-(m-2k+2k_2+k_3)=2k+2k_1-2k_2+1$. If this number is greater than 2, than it is not tempered because $\SO_{2a+b}/\SO_a\times \SO_{a+b}$ is not tempered when $b>2$. If this number is less or equal to $2$, since it is an odd number, we must have $2k+2k_1-2k_2+1=1$ which means that $k=k_2-k_1$. Since $0\leq k_1,k_2\leq k$, this is only possible when $k_1=0$ and $k_2=k$. In this case $k_3$ also needs to be $0$. Moreover, in this case the first component of \eqref{so equation} is just $\GL_k$ and hence it only breaks into one Borel orbit. This is exactly the orbit that satisfies Proposition \ref{main prop}(1). This proves Proposition \ref{main prop}. Lastly, in this case $L_i$ is not tempered for $i\neq i_0$ and therefore  we can remove the locally L-supercuspidal condition in Theorem \ref{main theorem}.

For Model 7,  $(G,H)=(\SO_{2k+2m},\SO_m\times \SO_{m+2k})$ and $(L,H_L)=(\SO_{2m+2}\times \GL_{1}^{k-1},\SO_{m}\times \SO_{m+2}\times \GL_{1}^{k-1})$. The calculation is very similar to the previous model and we will skip it here. As in the previous case, in this case $L_i$ is not tempered for $i\neq i_0$ and therefore  we can remove the locally L-supercuspidal condition in Theorem \ref{main theorem}.

For Model 8, $(G,H)=(\SO_{4n},\GL_{2n})$ and $(L,H_L)=(\GL_{2n},\GL_n\times \GL_n)$. When $n$ is even (resp. odd), $H=L$ is a Levi subgroup of the parabolic of $Q$ (resp. $H=\theta(L)$ where $\theta$ is the outer automorphism of $G$). We will consider the case when $n$ is even, the odd case is similar. We first consider the double coset $Q\back G/Q.$ By the Bruhat decomposition, this has $n+1$ many orbits. For $0\leq i\leq n+1$, the stabilizer in $Q$ is of the form
$$\begin{pmatrix}\GL_{2n-2i} & \ast & \ast & \ast\\ 0 & \GL_{2i} & 0 & \ast \\ 0 & 0 & \GL_{2i} & \ast \\ 0 & 0 & 0 & \GL_{2n-2i}  \end{pmatrix}.$$
If we break into $Q\back G/ H=Q\back G/L$-orbits, we just need to study the action of $\GL_{2i}$ on antisymmetric $2i\times 2i$-matrix for which the orbits are given by the rank of the matrix. As a result, the orbits of $Q\back G/ H=Q\back G/L$ are given by $(i,j)$ with $i$ as before and $0\leq j\leq i$. The projection of the stabilizer in $Q$ to $L$ is given by 
$$diag(g,g^\sharp),\;g\in \begin{pmatrix}\GL_{2n-2i} & 0 & \ast \\ 0& \GL_{2i-2j} & \ast \\ 0 & 0 &\Sp_{2j}   \end{pmatrix},\;g^\sharp=(g^t)^{-1}.$$
Unless $j=0$ and $i=n/2$, the stabilizer is not tempered because $\Sp_{2j}\subset \GL_{2j}$ (when $j\neq 0$) and $\GL_{2n-2i}\times \GL_{2i}\subset \GL_{2n}$ (when $j=0$) are not tempered. The only orbit left is when $j=0$ and $i=n/2$, this is exactly the orbit that satisfies Proposition \ref{main prop}(1). This proves Proposition \ref{main prop}. In this case $L_i$ is not tempered for $i\neq i_0$ and therefore  we can remove the locally L-supercuspidal condition in Theorem \ref{main theorem}.

For Model 9, $(G,H)=(\SO_{4n+2},\GL_{2n+1})$ and $(L,H_L)=(\GL_{2n}\times \GL_1,\GL_n\times \GL_n\times \GL_1)$. The calculation is very similar to the previous model and we will skip it here. As in the previous case, in this case $L_i$ is not tempered for $i\neq i_0$ and therefore  we can remove the locally L-supercuspidal condition in Theorem \ref{main theorem}.

For Model 10-12, Proposition \ref{main prop} has been proved in Lemma 2.1 of \cite{J}. Moreover, for Model 10, according to the computation in loc. cit., $L_i$ is not tempered for $i\neq i_0$ and therefore  we can remove the locally L-supercuspidal condition in Theorem \ref{main theorem}.

For Model 13, $(G,H)=(\GSO_{10},\GSpin_7)$ and $(L,H_L)=(\GSO_6\times \GL_2,\GL_2\times \GL_2)$. In this case, the double coset $Q\back G/H$ has been studied in Section 9 of \cite{PWZ}, and Proposition \ref{main prop} follows from Proposition 9.12, Lemma 9.13, Lemma 9.15 and Proposition 9.18 of \cite{PWZ2} \footnote{Here we refer the reader to the Arxiv version of \cite{PWZ2} in \url{https://arxiv.org/pdf/1903.02544} instead of the published version. The reason is that in the final published version, by using the relative truncation operator instead of the Arthur's truncation operator, we only need to consider the closed orbits the double coset $Q\back G/H$ and hence we do not include the computation of all the orbits in published version.} (with the notation as in loc. cit., the orbit $Q\gamma_{i_0}H$ in Proposition \ref{main prop}(1) corresponds to the case when the two dimensional isotropic subspace $W$ is contained in the Octonian algebral $\Theta$ and is null). 

For Model 14, $(G,H)=(\SO_9,\Spin_7)$ and $(L,H_L)=(\SO_5\times \GL_2,\Spin_3\times \GL_2)$. In this case, Proposition \ref{main prop} can be proved by a similar but easier argument as Model 13 \footnote{In Model 13, the orbits correspond to the action of $\Spin_7$ on the set of two dimensional isotropic subspaces of $\Theta\oplus k^2$; for Model 14, the orbits correspond to the action of $\Spin_7$ on the set of two dimensional isotropic subspaces of $\Theta\oplus k$ which is easier.} and we will skip it here.

For Model 15 (resp. Model 16), $(G,H)=(\GL_{2n+2}\times \GL_2,\GL_{2n}\times \GL_2)$ (resp. $(G,H)=(\GL_{2n+1}\times \GL_2,\GL_{2n-1}\times \GL_2)$) and $(L,H_L)=(\GL_4\times \GL_{1}^{2n-2}\times \GL_2,\GL_2\times \GL_2\times \GL_{1}^{2n-2})$ (resp. $(L,H_L)=(\GL_5\times \GL_{1}^{2n-4}\times \GL_2,\GL_3\times \GL_2\times \GL_{1}^{2n-4})$). The calculation of the double coset $Q\back G/H$ follows from the case of Model 1 when $a=2$ and $k=n-1$ (resp. Model 2 when $a=2$ and $k=n-2$), the only difference is that we need to view the stabilizer in $\GL_4$ (resp. $\GL_5$) computed in Model 1 as a subgroup of $\GL_4\times \GL_2$ (resp. $\GL_5\times \GL_2$). By the computation in Model 1 (resp. 2), for all but one orbit, the stabilizer will either contains $\GL_3\subset \GL_4$ (resp. $\GL_4\subset \GL_5$) which is not tempered, or contains a unipotent radical of a proper parabolic subgroup of $\GL_4$ (resp. $\GL_5$). The remaining orbit is precisely the one satisfying Proposition \ref{main prop}(1). This completes the proof of Proposition \ref{main prop}.

Before we consider the remaining models, we first study the double coset
\begin{equation}\label{sp double coset}
Q'\back \Sp_{2p+2}/\Sp_{2p}\times \Sp_2
\end{equation}
because we are going to use it for all but one remaining model. Here $Q'$ is the parabolic subgroup of $\Sp_{2p+2}$ whose Levi factor is $L'=\Sp_4\times \GL_{1}^{p-1}$. This is the same as the action of $Q'$ on the 2-dimensional nondegenerate subspaces $V_0$ of $V$ where $V$ is symplectic space defining $\Sp_{2p+2}$. Using the Levi subgroup $L'=\Sp_4\times \GL_{1}^{p-1}$ we can decompose $V$ as $W_4\oplus W_+\oplus W_-$ where $W_+,W_-$ are isotropic subspaces of dimension $p-1$ and $W_4$ is an anisotropic subspace of dimension $4$ (in particular $Q'$ preserves the isotropic subspace $W_+$). If the projection of $V_0$ on $W_-$ is two dimensional, we can use the unipotent elements in $Q'$ to make $V_0$ contained in $W_+\oplus W_-$ and the stabilizer in $Q'$ would contain the whole $\Sp_4$-part of $L'$. If the projection of $V_0$ on $W_-$ is one dimensional, we can use unipotent elements in $Q'$ to make $\dim(V_0\cap (W_-\oplus W_+))\geq 1$ which implies that the projection of $V_0$ on $W_4$ is at most one dimensional. Then the stabilizer in $Q'$ contains the mirabolic subgroup of $\Sp_4$. Note that the first two cases correspond to more than one orbit. Moreover, in these two cases the stabilizer all contains the mirabolic subgroup of $\Sp_4$ that acts trivially on vectors in $V_0$, which implies that it has trivial intersection with the $\Sp_2$-part of $\Sp_{2p}\times \Sp_2$. If the projection of $V_0$ on $W_-$ is zero, then we can use the unipotent elements in $Q'$ to make $V_0$ contained in $W_4$. Since $\Sp_4$ acts transitively on the set of 2-dimensional nondegenerate subspaces of $W_4$, there is only one orbit in this case and the projection of the stabilizer in $Q'$ to the $\Sp_4$-part of the Levi is $\Sp_2\times \Sp_2$.

Now we can consider the remaining models. For Model 17 (resp. Model 20, Model 23, Model 25), we have
$$(G,H)=(\Sp_{2p+2}\times \Sp_2,\Sp_2\times \Sp_{2p}),\;(L,H)=(\Sp_{4}\times \Sp_2\times \GL_{1}^{p-1},\Sp_2\times \Sp_{2}\times \GL_{1}^{p-1}),$$ 
$$\text{resp.}\;(G,H)=(\GL_{p+2}\times \Sp_{2q+2},\GL_p\times \SL_2\times \Sp_{2q}),\; (L,H_L)=\GL_{p+2}\times \Sp_{4}\times \GL_{1}^{q-1},\GL_p\times \SL_2\times \SL_2\times \GL_{1}^{q-1}), $$
$$(G,H)=(\Sp_4\times \Sp_{2p+2}\times \Sp_{2},\Sp_{2}^2\times \Sp_{2p}),\;(L,H_L)=(\Sp_{4}^{2}\times \Sp_{2}\times \GL_{1}^{p-1},\Sp_{2}^{3}\times \GL_{1}^{p-1}),$$
$$(G,H)=(\Sp_{2p+2}\times \Sp_{2}^{2},\Sp_{2p}\times \Sp_2),\; (L,H_L)=(\Sp_{4}\times \Sp_{2}^{2}\times \GL_{1}^{p-1},\Sp_{2}^{2}\times \GL_{1}^{p-1}).$$
It is clear that the double coset $Q\back G/H$ is in bijection with the double coset \eqref{sp double coset}. Moreover, by our computation of the stabilizers for the double coset \eqref{sp double coset}, we know that for all but one orbit, $L_i$ contains the mirabolic subgroup of $\Sp_4$ which is not tempered. It is easy to see that the remaining orbit is precisely the one that satisfies Proposition \ref{main prop}(1). This completes the proof of Proposition \ref{main prop}. In this case $L_i$ is not tempered for $i\neq i_0$ and therefore  we can remove the locally L-supercuspidal condition in Theorem \ref{main theorem}.

For Model 18, we have $(G,H)=(\Sp_{2p+2}\times \Sp_{2q+2},\Sp_2\times \Sp_{2p}\times \Sp_{2q})$ and $(L,H_L)=(\Sp_4\times \Sp_4\times \GL_{1}^{p+q-2},\Sp_{2}^{3}\times \GL_{1}^{p+q-2})$. Let $H'=\Sp_{2}^2\times \Sp_{2p}\times \Sp_{2q}$ be the subgroup of $G$ containing $H$. We first consider the double coset $Q\back G/H'$ and then breaks each orbit into $H$-orbits. The double coset $Q\back G/H'$ is just two copies of the double coset \eqref{sp double coset}. By our computation of the stabilizers for the double coset \eqref{sp double coset}, we know that for all but one orbits in $Q\back G/H'$, the projection of the stabilizer of $Q$ in $L$ would contain the mirabolic subgroup of at least one $\Sp_4$-copy. Moreover, such a mirabolic subgroup has trivial intersection with the $\Sp_{2}^{2}$-part of $H'$ (this is the only difference between $H$ and $H'$). As a result, when we breaks thoes orbits into $H$-orbits, the stabilizer $L_i$ would also contain the mirabolic subgroup of at least one $\Sp_4$-copy, which is not tempered. For the last orbit in $Q\back G/H'$, the stabilizer in $H'$ contains $\Sp_{2}^{2}$ and hence breaks into only one $H$-orbit. It is clear that this orbit satisfies Proposition \ref{main prop}(1). This proves Proposition \ref{main prop}. The same argument can also be used to prove Proposition \ref{main prop} for Model 24, 26 and 27, and we will skip the details here. For all these models, $L_i$ is not tempered for $i\neq i_0$ and therefore  we can remove the locally L-supercuspidal condition in Theorem \ref{main theorem}.

For Model 19, $(G,H)=(\Sp_{2p+4}\times \Sp_4,\Sp_{2p}\times \Sp_4)$ and $(L,H_L)=(\Sp_8\times \GL_{1}^{p-2}\times \Sp_4,\Sp_{4}^{2}\times \GL_{1}^{p-2})$ with $p>2$. The double coset $Q\back G/H$ is the same as the double coset
\begin{equation}\label{sp double 2}
Q_1\back \Sp_{2p+4}/\Sp_{2p}\times \Sp_4
\end{equation}
where $Q_1$ is the parabolic subgroup of $\Sp_{2p+4}$ whose Levi factor is $L_1=\Sp_8\times \GL_{1}^{p-2}$ (in particular $Q=Q_1\times \Sp_4$). We first consider the double coset
\begin{equation}\label{sp double 3}
Q_1'\back \Sp_{2p+4}/\Sp_{2p}\times \Sp_4
\end{equation}
where $Q_1'$ is the parabolic subgroup of $\Sp_{2p+4}$ whose Levi factor is $L_1'=\Sp_8\times \GL_{p-2}$. This is similar to the computation for Model 4. Let $V$ be the $2p+4$-dimensional symplectic space defining $\Sp_{2p+4}$ and $V=V_1\oplus V_2$ be the decomposition induced by the subspace $\Sp_{2p}\times \Sp_4\subset \Sp_{2p+4}$ with $\dim(V_1)=4$ and $\dim(V_2)=2p$. The double coset \ref{sp double 3} corresponds to the orbits of $\Sp_{2p}\times \Sp_4$ acts on the set of $(p-2)$-dimensional isotropic subspaces. Let $W\subset V$ be a $(p-2)$-dimensional isotropic subspace. As in Model 4, we let $k_i=\dim(W\cap V_i)$, $W_i$ be the projection of $W$ on $V_i$, and $2k_3$ be the dimension of the maximal anisotropic subspace of $W_i$. Similar to the computation in Model 4, the double cosets \eqref{sp double 3} are parametrized by the number $k_1,k_2,k_3$ that satisfy the following inequalities
$$0\leq k_1\leq 2, 0\leq k_2\leq p,p-6\leq k_2-k_1, k_1+k_2+2k_3\leq p-2,k_3\geq p-4-k_2.$$
It is easy to see that there are only 10 choices of $(k_1,k_2,k_3)$ satisfy the inequalities above:
$$(k_1,k_2,k_3)=(2,p-4,0),\;(1,p-3,0),\;(1,p-4,0),\;(1,p-5,1),\;(0,p-2,0),\;(0,p-3,0),$$
$$(0,p-4,0),\;(0,p-4,1),\;(0,p-5,1),\;(0,p-6,2)$$
where $(2,p-4,0),(1,p-4,0),(0,p-4,0),(0,p-4,1)$ (resp. $(1,p-5,1),(0,p-5,1)$) only appear when $p>3$ (resp. $p>4$) and $(0,p-6,2)$ only appears when $p>5$. We will only consider the case when $p>5$. The cases when $3\leq p\leq 5$ have less orbits and hence are easier.

For each $(k_1,k_2,k_3)$, the projection of the stabilizer in $Q_1'$ to $L_1'=\GL_{p-2}\times \Sp_8$ is of the form
\begin{equation}\label{sp comuptation 2}
\begin{pmatrix}\GL_{k_1} & 0 & \ast & \ast \\ 0 & \GL_{k_2} & \ast & \ast \\ 0 & 0 & \Sp_{2k_3} & \ast \\ 0 & 0 & 0 & \GL_{p-2-k_1-k_2-2k_3}^{diag} \end{pmatrix}
\end{equation}
$$\times \begin{pmatrix}\GL_{p-2-k_1-k_2-2k_3}^{diag} & \ast & \ast \\ 0 & \Sp_{2k_2+2k_3+8-2p}\times \Sp_{4+2k_1+2k_3} & \ast\\ 0 & 0 & (\GL_{p-2-k_1-k_2-2k_3}^{diag})^\sharp \end{pmatrix},\;g^\sharp=(g^t)^{-1}$$
and the stabilizer in $\Sp_{4}\times \Sp_{2p}$ is
\begin{equation}\label{sp computation 3}
\begin{pmatrix}\GL_{k_1} & \ast & \ast \\ 0 & \Sp_{2k_2+2k_3+8-2p}\times h^{diag}u & \ast \\ 0 & 0 & (\GL_{k_1})^{\sharp} \end{pmatrix}\times \begin{pmatrix}\GL_{k_2} & \ast & \ast \\ 0 & \Sp_{4+2k_1+2k_3}\times h^{diag} & \ast \\ 0 & 0 & (\GL_{k_2})^{\sharp} \end{pmatrix}
\end{equation}
with 
$$h\in \begin{pmatrix}\GL_{p-2-k_1-k_2-2k_3} & \ast & \ast \\ 0 & \Sp_{2k_3} & \ast \\ 0 & 0 & (\GL_{p-2-k_1-k_2-2k_3})^{\sharp}\end{pmatrix},u\in \begin{pmatrix}I_{p-2-k_1-k_2-2k_3} & 0 & \ast \\ 0 & I_{2k_3} & 0 \\ 0 & 0 & I_{p-2-k_1-k_2-2k_3}\end{pmatrix}.$$
When we future take the Borel orbit of the first component in \eqref{sp comuptation 2} to get the orbits for $Q_1\back \Sp_{2p+4}/\Sp_{2p}\times \Sp_4$, the projection of the stabilizer of each orbit in $Q_1$ to the $\Sp_8$-part of $L_1$ becomes
$$\begin{pmatrix}B_{p-2-k_1-k_2-2k_3} & \ast & \ast \\ 0 & \Sp_{2k_2+2k_3+8-2p}\times \Sp_{4+2k_1+2k_3} & \ast\\ 0 & 0 & (B_{p-2-k_1-k_2-2k_3})^\sharp \end{pmatrix}.$$
When $p-2-k_2-k_2-2k_3\neq 0$, this contains the unipotent radical of a proper parabolic subgroup of $\Sp_8$. Moreover, by our description of the stabilizer in $\Sp_{4}\times \Sp_{2p}$ we know that this unipotent subgroup has trivial intersection with $\Sp_4\times \Sp_{2p}$. As a result, for the corresponding orbits in the double coset $Q\back G/H$, the projection of the stabilizer in $Q$ to $L=\Sp_8\times \Sp_4$ will contain the unipotent radical of a proper parabolic subgroup. 

It remains to consider the case when $p-2-k_2-k_2-2k_3=0$. In this case, the projection of the stabilizer of the orbit in $Q_1$ to the $\Sp_8$-part of $L_1$ becomes
$$\Sp_{2k_2+2k_3+8-2p}\times \Sp_{4+2k_1+2k_3}.$$
If $k_1+k_3=2$, this stabilizer is $\Sp_8$ and has trivial intersection with the $\Sp_4$-part of the group $\Sp_4\times \Sp_{2p}$. As a result, for the corresponding orbits in the double coset $Q\back G/H$, the projection of the stabilizer in $Q$ to $L=\Sp_8\times \Sp_4$ will contain $\Sp_8$ which is not tempered. 

Now there are only three cases left, which are
$$(k_1,k_2,k_3)=(0,p-4,1),\;(1,p-3,0),\;(0,p-2,0).$$
When $(k_1,k_2,k_3)=(0,p-4,1)$, the projection of the stabilizer of the orbit in $Q_1$ to the $\Sp_8$-part of $L_1$ is $\Sp_6\times \Sp_2$. The first component of $\eqref{sp comuptation 2}$ becomes $diag(\GL_{p-4},\Sp_2)$. When we break it into Borel orbits, the stabilizer is of the form $diag(B_{p-4},B_{\Sp_2})$ where $B_{\Sp_2}$ is a Borel subgroup of $\Sp_2$. Combine with \eqref{sp computation 3}, the stabilizer in $\Sp_{4}\times \Sp_{2p}$ is
$$(\Sp_2\times B_{\Sp_2}^{diag})\times \begin{pmatrix}B_{p-4} & \ast & \ast \\ 0 & \Sp_{6}\times B_{\Sp_2}^{diag} & \ast \\ 0 & 0 & (B_{p-4})^{\sharp} \end{pmatrix}$$
As a result, for the corresponding orbits in the double coset $Q\back G/H$, the projection of the stabilizer in $Q$ to $L=\Sp_8\times \Sp_4$ is of the form 
$$(\Sp_6\times \Sp_{2}^{diag})\times (\Sp_{2}^{diag}\times B_{\Sp_2})$$
which is not tempered since $\Sp_6\times \GL_1\subset \Sp_8$ is not tempered.

When $(k_1,k_2,k_3)=(1,p-3,0)$, the projection of the stabilizer of the orbit in $Q_1$ to the $\Sp_8$-part of $L_1$ is $\Sp_6\times \Sp_2$. The first component of \eqref{sp comuptation 2} becomes $diag(\GL_{1},\GL_{p-3})$. When we break it into Borel orbits, the stabilizer is of the form $diag(\GL_1,B_{p-3})$. Combining with \eqref{sp computation 3}, the stabilizer in $\Sp_{4}\times \Sp_{2p}$ is
$$\begin{pmatrix}\GL_1 & \ast & \ast \\ 0 & \Sp_{2} & \ast \\ 0 & 0 & (\GL_1)^{\sharp} \end{pmatrix} \times \begin{pmatrix}B_{p-3} & \ast & \ast \\ 0 & \Sp_{6} & \ast \\ 0 & 0 & (B_{p-3})^{\sharp} \end{pmatrix}$$
As a result, for the corresponding orbits in the double coset $Q\back G/H$, the projection of the stabilizer in $Q$ to $L=\Sp_8\times \Sp_4$ is of the form 
$$(\Sp_6\times \Sp_{2}^{diag})\times \begin{pmatrix}\GL_1 & \ast & \ast \\ 0 & \Sp_{2}^{diag} & \ast \\ 0 & 0 & (\GL_1)^{\sharp} \end{pmatrix}$$
which is not tempered since $\Sp_6\times \Sp_2\subset \Sp_8\times \Sp_2$ is not tempered.

The last case is when $(k_1,k_2,k_3)=(0,p-2,0)$. In this case, the first component of \eqref{sp comuptation 2} is just $\GL_{p-2}$. Hence when we break it into Borel orbits, we only get one orbit. It is easy to see this is the orbit that satisfies Proposition \ref{main prop}(1). This proves Proposition \ref{main prop}.

For Model 21, $(G,H)=(\GL_{2p+2}\times \Sp_{2q+2},\GL_{2p}\times \SL_2\times \Sp_{2q})$ and $(L,H_L)=(\GL_{4}\times \Sp_{4}\times \GL_{1}^{2p+q-3},\GL_2\times \SL_2\times \SL_2\times \GL_{1}^{2p+q-3})$. Let $Q_0$ be the parabolic subgroup of $G$ containing $Q$ whose Levi factor is $L_0=\GL_{2p+2}\times \Sp_{4}\times \GL_{1}^{q-1}$. We first study the double coset $Q_0\back G/H$ and then break each orbit into $Q$-orbits. The double coset $Q_0\back G/H$ is in bijection with the double coset \eqref{sp double coset}. By our computation of the stabilizers for the double coset \eqref{sp double coset}, we know that for all but one orbis in $Q_0\back G/H$, the projection of the stabilizer of $Q_0$ in $L_0$ would contain the mirabolic subgroup of $\Sp_4$ and this mirabolic has trivial intersection with the $\Sp_2$-copy of $H$. When we break each of those orbits into $Q$-orbits, the projection of the stabilizer in $Q$ to $L$ would still contain the mirabolic subgroup of $\Sp_4$ which is not tempered. 

It remains to study the last orbit in $Q_0\back G/H$. For this orbit, the projection of the stabilizer in $Q_0$ to $L_0$ is $\GL_{2p}\times \SL_2\times \SL_2\times \GL_{1}^{q-1}$. When we further break it into $Q$-orbits, it is the same as the double coset $Q\back G/H$ for Model 15. Hence Proposition \ref{main prop} for the current model follows from Proposition \ref{main prop} for Model 15.

For Model 22, $(G,H)=(\GL_{2p+3}\times \Sp_{2q+2},\GL_{2p+1}\times \SL_2\times \Sp_{2q})$ and $(L,H_L)=(\GL_{5}\times \Sp_{4}\times \GL_{1}^{2p+q-3},\GL_3\times \SL_2\times \SL_2\times \GL_{1}^{2p+q-3})$. The proof of Proposition \ref{main prop} is similar to the previous case, the only difference is that we will be using Proposition \ref{main prop} of Model 16 instead of Model 15. We will skip the details here. 

This finishes the proof of Proposition \ref{main prop} for all the models in Table \ref{Table 1} and the proof of Theorem \ref{main theorem}.

\begin{rmk}
In this paper, we consider only one representative of the spherical varieties for each root type in Table \ref{Table 1}. However, it is possible to consider other forms; for example, one might replace the general linear group (resp.\ the special orthogonal group) by the special linear group (resp.\ the Spin group). When working with these alternative forms, the structure of the double coset space $Q \backslash G/H$ may change over a number field (over the algebraic closure the double coset remains the same), as one double coset over algebraic closure may split into multiple (or even infinitely many) orbits over a number field.

Nevertheless, Proposition \ref{main prop} (and hence Theorem \ref{main theorem}) remains valid for two reasons. First, for the double coset in Proposition \ref{main prop}(1), the stabilizer in $H$ is a parabolic subgroup. Since the map $H^1(k, Q_H) \to H^1(k, H)$ is injective for any parabolic subgroup $Q_H$ of $H$, this double coset does not split when passing to a number field. Second, for the double coset in Proposition \ref{main prop}(2), while they may split into multiple double coset, this does not affect the non-temperedness or the parabolically induced type condition, as these properties depend only on the spherical variety over the algebraic closure.
\end{rmk}

\end{document}